\documentclass[11pt,letterpaper]{article}
\usepackage[utf8]{inputenc}
\usepackage{tikz}
\usepackage{mathtools}
\usepackage{tikz-cd}
\usepackage{bm}
\usepackage{mathrsfs}
\usepackage{cite}
\usepackage[small,nohug,heads=LaTeX]{diagrams}
\usepackage{amsmath}
\numberwithin{equation}{section}
\usepackage{hyperref}
\usepackage{url}
\usepackage{minitoc}
\usepackage{lineno}
\usepackage{mathrsfs}
\usepackage{amsfonts}
\usepackage{amsthm}
\usepackage{amssymb}
\theoremstyle{definition}
\newtheorem{define}{Definition}[section]
\newtheorem{example}[define]{Example}
\newtheorem{construction}{Construction}[section]
\theoremstyle{remark}
\newtheorem{remark}[define]{Remark}
\theoremstyle{plain}
\newtheorem{theo}[define]{Theorem}
\newtheorem*{theorem*}{Theorem}
\newtheorem{lemma}[define]{Lemma}
\newtheorem{prop}[define]{Proposition}
\newtheorem{cor}[define]{Corollary}
\newcommand{\X}{\mathscr X}
\newcommand{\C}{\mathscr C}
\newcommand{\E}{\mathscr E}
\newcommand{\G}{\mathscr G}
\newcommand{\A}{\mathscr A}

\newcommand{\D}{\mathscr D}

\newcommand{\U}{\mathscr U}
\newcommand{\V}{\mathscr V}
\newcommand{\W}{\mathscr W}

\newcommand{\dom}{\mathrm{dom}}
\newcommand{\cod}{\mathrm{cod}}
\newcommand{\stab}{\mathbf{Stab}}

\usepackage{amssymb}
\usepackage[left=2cm,right=2cm,top=2cm,bottom=2cm]{geometry}

\begin{document}

\title{A 2-Site for Continuous 2-Group Actions}

\author{Michael Lambert}

\maketitle

\begin{abstract}  Elmendorf's Theorem states that the category of continuous actions of a topological group is a Grothendieck topos in the sense that it is equivalent to a category of sheaves on a site.  This paper offers a 2-dimensional generalization by showing that a certain 2-category of continuous actions of a topological 2-group is 2-equivalent to a 2-category of 2-sheaves on a suitable 2-site.
\end{abstract}

\tableofcontents

\section{Introduction}

The main result of this paper is a 2-dimensional version of A. D. Elmendorf's theorem in \cite{Elmendorf} showing that the category of continuous actions of a topological group on discrete spaces is a Grothendieck topos.  The present purpose is to show that a certain 2-category of continuous actions of a topological 2-group on discrete 2-spaces is a ``Grothendieck 2-topos" essentially as in \cite{StreetSheafThy}.  Only strict versions of the definitions of 2-group, action, and stabilizer are considered here.  A follow-up paper will study the possibility of giving an analogous ``bicategorical" result for the more general notions of ``coherent" or perhaps ``weak" 2-groups as in \cite{Baez2Groups}.

The version of Elmendorf's Theorem followed here is the account in \S III.9 of \cite{MM}.  More precisely, this shows that the category $\mathbf BG$ of continuous actions of a topological group $G$ on discrete spaces with equivariant maps between them is equivalent to a category of sheaves on a Grothendieck site.  The underlying category of the site is basically the so-called ``orbit category" associated to $G$, consisting of certain coset spaces taken over the open subgroups of $G$, together with the ``atomic topology" in which covering sieves are generated by the singletons.  The generalization in this paper is the following, which appears as Theorem \ref{MAIN THEOREM}.

\begin{theorem*} For any topological 2-group $\G$, there is a 2-category $\mathfrak S(\G)$ as yielding a 2-equivalence
\[ \mathfrak{Sh}(\mathfrak S(\G),J_{at})\simeq \mathfrak B\G
\]
between the 2-category $\mathfrak{Sh}(\mathfrak S(\G),J_{at})$ of 2-sheaves on $\mathfrak S(\G)$ for the atomic topology and the 2-category $\mathfrak B\G$ of continuous actions of $\G$ on discrete 2-spaces.
\end{theorem*}

The important contents of the paper are summarized as follows.  In \S 2.2 there is on offer a definition of a strict continuous action of a strict topological 2-group  $\G$ on a (discrete) 2-space.  Also proposed is a definition of an open sub-2-group of $\G$ that will fit with the rest of the account.  In \S 4.2 the 2-site $\mathfrak S(\G)$ is constructed as a sort of orbit 2-category of $\G$.  This can be provided with the atomic topology on its underlying 1-category.  In \S 4.4 the 2-category of 2-sheaves on this 2-site is seen in the main result of the paper -- namely, Theorem \ref{MAIN THEOREM} -- to be 2-equivalent to the 2-category of continuous $\G$-actions.

The rest of this introduction is concerned with giving some background on 2-categories and bicategories needed throughout the paper.  The goal is to be fairly explicit and self-contained so as to serve the backgrounds and interests of a diverse group of potential readers.

\subsection{2-Categories and Bicategories}

Throughout $\mathbf{Cat}$ denotes the ordinary 1-category of small categories and functors between them.  The basic viewpoint is that a 2-category is a $\mathbf{Cat}$-category, that is, a category enriched in $\mathbf{Cat}$ in the sense of \cite{Enriched}.  The reference \cite{KS} gives a more elementary description of what this means.  Roughly, a 2-category $\mathfrak A$ is a category with further 2-cells $\theta\colon f\Rightarrow g$ between arrows $f,g\colon A\rightrightarrows B$, two operations of ``vertical" and ``horizontal" composition of 2-cells, and finally suitable identity 2-cells for the compositions.  Horizontal composition is denoted with `$\ast$' whereas all other compositions are denoted by juxtaposition.  The 2-data satisfies a number of expected axioms, including a certain interchange law relating horizontal and vertical composition of 2-cells.  The basic example is $\mathfrak{Cat}$, the 2-category of small categories, functors between them, and natural transformations.  Horizontal composition of natural transformations
$$\begin{tikzpicture}
\node(1){$\C$};
\node(2)[node distance=.5in, right of=1]{$\Downarrow \alpha$};
\node(3)[node distance=.5in, right of=2]{$\D$};
\node(4)[node distance=.5in, right of=3]{$\Downarrow \beta$};
\node(5)[node distance=.5in, right of=4]{$\E$};
\draw[->, bend left=70](1) to node [above]{$F$}(3);
\draw[->, bend right=70](1) to node [below]{$G$}(3);
\draw[->, bend left=70](3) to node [above]{$H$}(5);
\draw[->, bend right=70](3) to node [below]{$K$}(5);
\end{tikzpicture}$$
is given on components $C\in \C_0$ by the formula
\begin{equation} \label{horiz comp in cat} (\beta\ast\alpha)_C:= \beta_{GC}H(\alpha_C) = K(\alpha_C)\beta_{FC}.
\end{equation}

Denote the sets of objects, arrows and 2-cells of a small 2-category by $\mathfrak A_0$, $\mathfrak A_1$ and $\mathfrak A_2$, respectively.  Every 2-category $\mathfrak A$ has an underlying 1-category $|\mathfrak A|$ consisting only of the objects and arrows.  Every 1-category can be viewed as a ``locally discrete" 2-category with only identity 2-cells between any two morphisms with the same domain and the same codomain.  For any 2-category $\mathfrak A$, the 1-dimensional opposite $\mathfrak A^{op}$ is obtained by formally reversing the morphisms, but not the 2-cells.

A 2-functor $F\colon \mathfrak A\to\mathfrak B$ is a morphism of 2-categories $\mathfrak A$ and $\mathfrak B$.  Roughly, it is a functor of underyling 1-categories $|F|\colon |\mathfrak A|\to |\mathfrak B|$ together with an assignment on 2-cells respecting the two compositions and their identities.  A 2-natural transformation $\theta\colon F\Rightarrow G$ between 2-functors $F,G\colon \mathfrak A\rightrightarrows \mathfrak B$ assigns to each $A\in \mathfrak A_0$ an arrow $\theta_A\colon FA\to GA$ such that for each $f\colon A\to A'$ the usual naturality square
$$\begin{tikzpicture}
\node(1){$FA$};
\node(2)[node distance=1in, right of=1]{$GA$};
\node(3)[node distance=.7in, below of=1]{$FA'$};
\node(4)[node distance=.7in, below of=2]{$GA'$};
\node(5)[node distance=.5in, right of=1]{$$};
\node(6)[node distance=.35in, below of=5]{$=$};
\draw[->](1) to node [above]{$\theta_A$}(2);
\draw[->](1) to node [left]{$Ff$}(3);
\draw[->](2) to node [right]{$Gf$}(4);
\draw[->](3) to node [below]{$\theta_{A'}$}(4);
\end{tikzpicture}$$
commutes and the following compatibility condition is satisfied:
\begin{enumerate}
\item[]{$\textbf{2-Natural Compatibility.}$} \label{2-nat compat} There is an equality in $\mathfrak B$ of composite 2-cells
$$\begin{tikzpicture}
\node(1){$FA$};
\node(2)[node distance=1.2in, right of=1]{$GA$};
\node(3)[node distance=1in, below of=1]{$FA'$};
\node(4)[node distance=1in, below of=2]{$GA'$};
\node(5)[node distance=.7in, right of=1]{$$};
\node(6)[node distance=.5in, below of=5]{$$};
\node(7)[node distance=2in, right of=2]{$FA$};
\node(8)[node distance=1.2in, right of=7]{$GA$};
\node(9)[node distance=1in, below of=7]{$FA'$};
\node(10)[node distance=1in, below of=8]{$GA'$};
\node(11)[node distance=1.5in, right of=6]{$=$};
\node(12)[node distance=1.5in, right of=11]{$$};
\node(13)[node distance=.4in, below of=1]{$F\alpha$};
\node(14)[node distance=.54in, below of=1]{$\Rightarrow$};
\node(15)[node distance=.4in, below of=8]{$G\alpha$};
\node(16)[node distance=.54in, below of=8]{$\Rightarrow$};
\draw[->](1) to node [above]{$\theta_A$}(2);
\draw[->,bend right=40](1) to node [left]{$Ff$}(3);
\draw[->,bend left=40](1) to node [right]{$Fg$}(3);
\draw[->](2) to node [right]{$Gg$}(4);
\draw[->](3) to node [below]{$\theta_{A'}$}(4);
\draw[->](7) to node [above]{$\theta_A$}(8);
\draw[->](7) to node [left]{$Ff$}(9);
\draw[->,bend right=40](8) to node [left]{$Gf$}(10);
\draw[->,bend left=40](8) to node [right]{$Gg$}(10);
\draw[->](9) to node [below]{$\theta_{A'}$}(10);
\end{tikzpicture}$$
for each 2-cell $\alpha \colon f\Rightarrow g$ of $\mathfrak A$ between arrows $f,g\colon A\rightrightarrows A'$.
\end{enumerate}
A 2-natural transformation $\theta$ is a 2-natural isomorphism if each component $\theta_A$ is an isomorphism in $\mathfrak B$.  

A final layer of structure is given in the notion of a modification, originating in \cite{Benabou}.  This is a morphism $m\colon \theta\Rrightarrow \gamma$ of 2-natural transformations consisting of, for each $A\in \mathfrak A_0$, a 2-cell $m_A\colon \theta_A \Rightarrow \gamma_A$ satisfying the following compatibility condition:
\begin{enumerate}
\item[]{\textbf{Modification Condition.}} \label{modification condition} There is an equality of 2-cells
$$\begin{tikzpicture}
\node(1){$FA$};
\node(2)[node distance=1.6in, right of=1]{$GA$};
\node(3)[node distance=.8in, below of=1]{$FA'$};
\node(4)[node distance=.8in, below of=2]{$GA'$};
\node(5)[node distance=.8in, right of=1]{$\Uparrow m_A$};
\node(6)[node distance=.9in, below of=5]{$$};
\node(7)[node distance=.5in, below of=5]{$$};
\node(8)[node distance=1.7in, right of=7]{$=$};
\node(9)[node distance=1.75in, right of=2]{$FA$};
\node(10)[node distance=1.6in, right of=9]{$GA$};
\node(11)[node distance=.8in, below of=9]{$FA'$};
\node(12)[node distance=.8in, below of=10]{$GA'$};
\node(13)[node distance=.8in, right of=9]{$$};
\node(14)[node distance=.2in, below of=13]{$$};
\node(15)[node distance=.6in, below of=14]{$\Uparrow m_{A'}$};
\draw[->](2) to node [right]{$Gf$}(4);
\draw[->](1) to node [left]{$Ff$}(3);
\draw[->,bend left](1) to node [above]{$\gamma_A$}(2);
\draw[->,bend right](1) to node [below]{$\theta_A$}(2);
\draw[->,bend right](3) to node [below]{$\theta_{A'}$}(4);
\draw[->,bend left](9) to node [above]{$\gamma_A$}(10);
\draw[->](9) to node [left]{$Ff$}(11);
\draw[->,bend left](11) to node [above]{$\gamma_{A'}$}(12);
\draw[->,bend right](11) to node [below]{$\theta_{A'}$}(12);
\draw[->](10) to node [right]{$Gf$}(12);
\end{tikzpicture}$$
for each arrow $f\colon A\to A'$ of $\mathfrak A$.
\end{enumerate}
Throughout $[\mathfrak A,\mathfrak B]$ denotes the 2-category of 2-functors $\mathfrak A\to\mathfrak B$, their 2-natural transformations, and modifications between them.  In particular, $[\mathfrak A^{op},\mathfrak{Cat}]$ will be considered the appropriate 2-dimensional analogue of the ordinary category of presheaves on a small 1-category.  Throughout adopt a notational convention of \cite{MM} used for presheaves by denoting the action $Ff(X) = X\cdot f$ and similarly on morphisms, where $f\colon A\to B$ is an arrow of $\mathfrak A$ and $F\colon \mathfrak A^{op}\to \mathfrak{Cat}$ is a 2-functor.

Two 2-categories $\mathfrak A$ and $\mathfrak B$ are considered to be 2-equivalent if they are equivalent in the sense of $\mathbf{Cat}$-enriched categories as in \S 1.11 of \cite{Enriched}.  This is spelled out more explicitly in the following statment. 

\begin{define} \label{equiv of 2-cats} The 2-categories $\mathfrak A$ and $\mathfrak B$ are \textbf{2-equivalent} if there are 2-functors $F\colon \mathfrak A \rightleftarrows\mathfrak B\colon G$ and 2-natural isomorphisms $\eta\colon 1\cong GF$ and $\epsilon\colon FG\cong 1$.
\end{define}

It is worth recalling briefly the definitions of 2-limits and 2-colimits.  In particular, as in \S 3.1 of \cite{Enriched}, the 2-colimit of a 2-functor $F\colon \mathfrak J\to\mathfrak K$ weighted by a 2-functor $W\colon \mathfrak J^{op}\to\mathfrak{Cat}$ is an object $W\star F$ of $\mathfrak K$ fitting into an isomorphism of categories
\begin{equation} \label{2-colimit}
\mathfrak K(W\star F,A)\cong [\mathfrak J^{op},\mathfrak{Cat}](W,\mathfrak{Cat}(F,A))
\end{equation}
where $\mathfrak{Cat}(F,A)$ denotes the 2-functor given by $X\mapsto \mathfrak{Cat}(FX,A)$ for $X\in \mathfrak J_0$ and extended suitably to morphisms and 2-cells of $\mathfrak J$.  The weight $X\mapsto \mathbf 1$ is the ``trivial weight," denoted again by $\mathbf 1$.  A 2-colimit with trivial weight is a ``conical" colimit -- basically a ``boosted up" 1-dimensional colimit with a 2-dimensional aspect to its universal property.

\begin{remark}  One sometimes hears murmurings that the theory of $\mathcal V$-categories is an elaboration of 1-dimensional category theory and not properly 2-dimensional.  Insofar as this is the case, there is some question as to whether a proper 2-dimensional generalization of Elmendorf's Theorem can really be achieved with the machinery of $\mathbf{Cat}$-enriched theory recalled so far.  Not only that, but the present account applies only to strict 2-groups with their strict monoidal structure.  Thus, recall that a bicategory is like a 2-category where composition of morphisms is associative up to coherent isomorphism.  The precise details are not important here and the full definition can be found for example in the original source, namely, \cite{Benabou}.  Instead a fully 2-dimensional redevelopment of Elmendorf's theorem might consider coherent 2-groups in the sense of \cite{Baez2Groups} and construct a bicategory of ``bisheaves" or stacks to which a bicategory of continuous actions in an ``up-to-isomorphism" sense would be biequivalent.  Some remarks in the conclusion of this paper will discuss the possibility of such a result.  For the moment, the $\mathbf{Cat}$-enriched version of this paper certainly holds for strict 2-groups.  The result seems to be of some intrinsic interest, but also shows where there arise certain well-definition problems that can be solved in the strict case but likely not in the coherent case.
\end{remark}

\section{Continuous 2-Group Actions}

Here is recalled the notion of a strict topological 2-group and given the definition of a continuous action on a category viewed as a discrete 2-space.  For undefined categorical notions see \cite{MacLane}.

\subsection{2-Groups}

Recall that a group object in a finitely-complete category is an object of the category together with a group law morphism, an identity point and an inverse morphism, all satisfying diagrammatic versions of the usual group laws.  A category object in a finitely-complete category consists of an object of objects, an object of arrows, various domain, codomain, composition and inverse arrows, all satisfying diagrammatic versions of the usual category axioms.   

The pithy definition of a strict 2-group is that it is a group object in the category of small categories.  But there are several well-known equivalent formulations.  For example, it is well-established that 2-groups are essentially the same thing as so-called ``crossed modules" as introduced in \cite{Whitehead}.  This is basically Theorem 1 of \cite{BrownSpencer1}.  But in a more abstract vein, it follows from the pithy definition that a 2-group $\G$ is also a groupoid.  This seems first to have been proved in \cite{BrownSpencer1} as well.  So, a 2-group could be defined as a group object in groupoids.  By finite-limit arguments, a 2-group could also be seen as a category object in groups, or even a groupoid object in groups.  The present development has settled around the category-object perspective.

\begin{define} \label{2-group def}  A \textbf{2-Group} is a category object in the category of groups.
\end{define}

This means that a 2-group $\G$ is a pair of groups $\G_0$ and $\G_1$ organized into a category.  In particular, $\G$ has an underlying category and the group laws $\otimes_0$ and $\otimes_1$ coming with $\G_0$ and $\G_1$ yield a tensor bifunctor $\otimes\colon \G\times\G\to\G$ with a distinguished unit object $I\in \G_0$ with distinguished identity arrow $1_I$.  That the tensor is a bifunctor amounts to a so-called ``internchange law," namely, that
\begin{equation}
\label{interchange} (g\otimes h)(k\otimes l) = gk\otimes hl
\end{equation}  
holds for suitably composable morphisms $g$ and $k$ on the one hand and $h$ and $l$ on the other.  The tensor and unit make $\G$ a strict monoidal category as in \S VII.1 of \cite{MacLane} in which every object and arrow has an inverse under the tensor.  Write the inverses of objects $A\in \G_0$ and arrows $g\in\G_1$ as $A^{-1}$ and $g^{-1}$, respectively.  Given an arrow $g\colon A\to B$ of a 2-group $\G$, the compositional inverse is 
\begin{equation}\label{comp inverse} \bar g = 1_A\otimes g^{-1} \otimes 1_B
\end{equation}
as can be seen using the interchange law \ref{interchange}.

\begin{remark} Definition \ref{2-group def} gives a ``strict" version of the idea of a 2-dimensional group.  Such $\G$ is a strict monoidal category.  There are various ways of weakening the axioms.  One is the notion of a ``coherent" 2-group as described in \S 3 of \cite{Baez2Groups}.  Roughly speaking, a coherent 2-group $\G$ is a weak monoidal category in which every morphism is invertible and such that every object $A$ is equipped with an adjoint equivalence $\iota_A\colon I\to A\otimes \bar A$ and $\epsilon_A\colon \bar A\otimes A\to I$.  However, in the present account,``2-group" will always mean one in the sense of Definition \ref{2-group def}.  
\end{remark}

The point of the phrasing of Definition \ref{2-group def} is that it makes it easy to define categories of 2-groups with extra structure.  Recall that a topological group is a group object in topological spaces.

\begin{define} \label{top 2-group}  A \textbf{topological 2-group} is a category object in the category of topological groups.
\end{define}

Thus, a 2-topological 2-group $\G$ consists of two topological groups $\G_0$ and $\G_1$ together with continuous group homomorphisms
$$\begin{tikzpicture}
\node(1){$\G_2$};
\node(2)[node distance=1.2in, right of=1]{$\G_1$};
\node(3)[node distance=1.2in, right of=2]{$\G_0$};
\draw[transform canvas={yshift=1.7ex},->](1) to node [above]{$\pi_1$}(2);
\draw[transform canvas={yshift=-1.7ex},->](1) to node [below]{$\pi_2$}(2);
\draw[->](1) to node [fill=white]{$\circ$}(2);
\draw[transform canvas={yshift=1.7ex},->](2) to node [above]{$d_0$}(3);
\draw[transform canvas={yshift=-1.7ex},->](2) to node [below]{$d_1$}(3);
\draw[->](3) to node [fill=white]{$i$}(2);
\end{tikzpicture}$$
satisfying the usual category axioms.  The group laws $\otimes_0\colon \G_0\times\G_0\to\G_0$ and $\otimes_1\colon \G_1\times\G_1\to\G_1$ are the underlying functors of the induced monoidal tensor $\otimes\colon \G\times\G\to\G$.  

\begin{remark} The idea of a topological coherent 2-group is made precise in \S 7 of \cite{Baez2Groups}.  Following Definition 26 of the reference, a topological 2-group is taken to be a ``2-group object" in the 2-category of category objects in, for example, topological spaces or perhaps $k$-spaces.  For now, however, we stick to the strict notion of Definition \ref{top 2-group}.
\end{remark}

\begin{define} \label{top sub 2-group} A \textbf{topological sub-2-group} of a topological 2-group $\G$ is a subcategory $\mathscr U\subset \G$ that is a 2-group under the operations of $\G$ and such that $\mathscr U_0$ and $\mathscr U_1$ are open in $\G_0$ and $\G_1$, respectively.  \end{define}

\begin{remark}  This definition of an open sub-2-group fits with the approach that a topological 2-group is two topological groups organized into a category with suitable continuous group homomorphisms.  Additionally, each stabilizer as in Example \ref{stabilizer} below is an open sub-2-group in this sense.  However, this definition might be too strict.  For example, the notion of the stabilizer of an action as defined by a \emph{pseudo}-pullback as in Remark \ref{stabilizer pseudo pb} below is not a sub-2-group in this sense.  It seems, rather, what is needed in this case is that the ``inclusion" $\mathscr U\to \G$ is merely a faithful functor.  This point will be addressed in future work. 
\end{remark}

The open sub-2-groups of a topological 2-group $\G$ as in Definition \ref{top sub 2-group} comprise a poset category under inclusion.  Denote this category by $\mathbf L(\G)$.  Usually $\mathbf L(\G)$ will be viewed as a locally discrete 2-category.

\subsection{Actions}

Actions of (coherent) 2-groups on so-called (smooth) ``2-spaces" were studied in Bartel's Thesis \cite{BartelsThesis}.  These were axiomatized in an ``up-to-isomorphism" sense using a certain codescent condition.  Strict actions of 2-groups on categories have been considered more recently, for example, in \S 3 of \cite{MortonPicken} and applied in \cite{MortonPickenGauge}.  These deal with strict 2-groups and what might be thought of a strict actions.  These notions can be easily topologized using the diagrammatic phrasing of their definitions. 

As set up, first recall that for an ordinary topological group $G$, a continuous action on a set $X$, provided with the discrete topology, is a continuous function $m\colon X\times G\to X$ satisfying the usual associativity and unit conditions.  A morphism of continuous actions $X$ and $Y$ is a function $f\colon X\to Y$ that is equivariant in the sense that 
$$\begin{tikzpicture}
\node(1){$X\times G$};
\node(2)[node distance=1in, right of=1]{$X$};
\node(3)[node distance=.7in, below of=1]{$Y\times G$};
\node(4)[node distance=.7in, below of=2]{$Y$};
\node(5)[node distance=.5in, right of=1]{$$};
\node(6)[node distance=.35in, below of=5]{$=$};
\draw[->](1) to node [above]{$m$}(2);
\draw[->](1) to node [left]{$1\times f$}(3);
\draw[->](2) to node [right]{$f$}(4);
\draw[->](3) to node [below]{$n$}(4);
\end{tikzpicture}$$
commutes.  Throughout use $\mathbf BG$ to denote the category of continuous right $G$-actions.

As in, for example, \S V.6 of \cite{MM}, the axioms for such actions make sense in any category with finite limits.  So, for the following, consider a given small category $\X$ as a category object in topological spaces with the sets of objects and morphisms each provided with the discrete topology.  In this way, the category $\X$ is viewed as a ``discrete 2-space."

\begin{define} \label{action}  A \textbf{right action} of a 2-group $\G$ on a category $\X$ is a functor $M\colon \X\times \G\to\X$ satisfying diagrammatic versions of the usual action axioms.  If $\G$ is a topological 2-group, an action is continuous if each of the underlying set functions $M_0\colon \X_0\times \G_0\to\X_0$ and $M_1\colon \X_1\times \G_1\to\X_1$ is continuous.
\end{define}

Such actions, together with $\G$-equivariant morphisms and 2-cells between them satisfying a compatibility condition, form a 2-category, denoted by $\mathfrak B\G$.  The condition is the following:
\begin{enumerate}
\item[]{\textbf{Action 2-Cell Compatibility.}} \label{action 2-cell compat}  There is an equality of 2-cells 
$$\begin{tikzpicture}
\node(1){$\X\times \G$};
\node(2)[node distance=1.2in, right of=1]{$\X$};
\node(3)[node distance=1in, below of=1]{$\mathscr Y\times \G$};
\node(4)[node distance=1in, below of=2]{$\mathscr Y$};
\node(5)[node distance=.7in, right of=1]{$$};
\node(6)[node distance=.5in, below of=5]{$$};
\node(7)[node distance=2in, right of=2]{$\X\times \G$};
\node(8)[node distance=1.2in, right of=7]{$\X$};
\node(9)[node distance=1in, below of=7]{$\mathscr Y\times \G$};
\node(10)[node distance=1in, below of=8]{$\mathscr Y$};
\node(11)[node distance=1.5in, right of=6]{$=$};
\node(12)[node distance=1.5in, right of=11]{$$};
\node(13)[node distance=.4in, below of=1]{$\theta\times i$};
\node(14)[node distance=.54in, below of=1]{$\Rightarrow$};
\node(15)[node distance=.4in, below of=8]{$\theta$};
\node(16)[node distance=.54in, below of=8]{$\Rightarrow$};
\draw[->](1) to node [above]{$M$}(2);
\draw[->,bend right=40](1) to node [left]{$H\times 1$}(3);
\draw[->,bend left=40](1) to node [right]{$K\times 1$}(3);
\draw[->](2) to node [right]{$K$}(4);
\draw[->](3) to node [below]{$N$}(4);
\draw[->](7) to node [above]{$M$}(8);
\draw[->](7) to node [left]{$H\times 1$}(9);
\draw[->,bend right=40](8) to node [left]{$H$}(10);
\draw[->,bend left=40](8) to node [right]{$K$}(10);
\draw[->](9) to node [below]{$N$}(10);
\end{tikzpicture}$$
for discrete 2-spaces $\X$ and $\mathscr Y$ with given continuous actions $M\colon \X\times\G\to\X$ and $N\colon \mathscr Y\times \G\to\mathscr Y$ as in Definition \ref{action} and $\G$-equivariant morphisms $H,K\colon \X\rightrightarrows\mathscr Y$.
\end{enumerate}  

Throughout write $X\cdot A = M(X,A)$ and similarly on arrows to cut down on notation.  Notice that, with $X\in \X_0$ fixed, there is an induced functor
\[ X\cdot (-)\colon \G\to\X
\]
given by $A\mapsto X\cdot A$ on objects and by $f\mapsto 1_X\cdot f$ on morphisms.  This is continuous at the level of objects and at the level of morphisms.  Similarly, for each morphism $m\colon X\to Y$ of $\X$, there is a functor 
\[ m\cdot(-)\colon \G\to\X^{\mathbf 2}
\]
where $\X^{\mathbf 2}$ denotes the arrow category of $\X$.  This functor is given by $A\mapsto m\cdot 1_A$ on objects and by $g \mapsto m\cdot g$ viewed as a square in the arrow category using the bifunctor condition for the action.

\begin{remark} The appropriate ``up to isomorphism" version of the definition, roughly speaking, will be the following.  An action of a coherent 2-group $\G$ on a category $\X$ is a functor $M\colon \X\times \G\to \X$ together with an isomorphism
$$\begin{tikzpicture}
\node(1){$\X\times \G\times \G$};
\node(2)[node distance=1.4in, right of=1]{$\X\times \G$};
\node(3)[node distance=.7in, below of=1]{$\X\times \G$};
\node(4)[node distance=.7in, below of=2]{$\X$};
\node(5)[node distance=.7in, right of=1]{$$};
\node(6)[node distance=.35in, below of=5]{$\cong$};
\draw[->](1) to node [above]{$M\times 1$}(2);
\draw[->](1) to node [left]{$1\times \otimes$}(3);
\draw[->](2) to node [right]{$M$}(4);
\draw[->](3) to node [below]{$M$}(4);
\end{tikzpicture}$$
satisfying a ``codescent" coherence condition for associativity and a unit condition, formally resembling the pseudo-algebra coherence laws of \cite{LackCodescent}.  This is basically the approach taken in Bartel's Thesis \cite{BartelsThesis} for actions on 2-spaces axiomatized as category objects in smooth manifolds.
\end{remark}

\begin{example}[Stabilizer Sub-2-Groups] \label{stabilizer} The first pass on the definition of the stabilizer of a morphism $m\in \X_1$ under a (continuous) action $M\colon \X\times\G\to\G$ is that it is the category with objects $A\in\G_0$ such that $m\cdot 1_A=m$ holds and arrows those $g\in \G_1$ such that $m\cdot g=m$ holds.  Note that this computes the (strict) pullback
$$\begin{tikzpicture}
\node(1){$\stab(m)$};
\node(2)[node distance=1in, right of=1]{$\mathbf 1$};
\node(3)[node distance=.7in, below of=1]{$\G$};
\node(4)[node distance=.7in, below of=2]{$\X^{\mathbf 2}$};
\node(5)[node distance=.25in, right of=1]{$$};
\node(6)[node distance=.2in, below of=5]{$\lrcorner$};
\draw[->](1) to node [above]{$$}(2);
\draw[->](1) to node [left]{$$}(3);
\draw[->](2) to node [right]{$m$}(4);
\draw[->](3) to node [below]{$m\cdot -$}(4);
\end{tikzpicture}$$
taken in $\mathfrak{Cat}$.  The stabilizer of an object $X\in\X_0$ has as objects those $A\in\G_0$ such that $X\cdot A=X$ holds and as arrows those $g\in \G_1$ with $1_X\cdot g=1_X$.  This computes the strict pullback
$$\begin{tikzpicture}
\node(1){$\stab(X)$};
\node(2)[node distance=1in, right of=1]{$\mathbf 1$};
\node(3)[node distance=.7in, below of=1]{$\G$};
\node(4)[node distance=.7in, below of=2]{$\X$};
\node(5)[node distance=.25in, right of=1]{$$};
\node(6)[node distance=.2in, below of=5]{$\lrcorner$};
\draw[->](1) to node [above]{$$}(2);
\draw[->](1) to node [left]{$$}(3);
\draw[->](2) to node [right]{$X$}(4);
\draw[->](3) to node [below]{$X\cdot -$}(4);
\end{tikzpicture}$$
taken in $\mathfrak{Cat}$.  Each stabilizer is an open sub-2-group in the sense of Definition \ref{top sub 2-group} if the action is continuous.  Notice that $\stab(X)\cong \stab(1_X)$ holds for each object $X\in\X_0$.
\end{example}

\begin{remark} \label{stabilizer pseudo pb} For the present purposes, the definition above is fine.  However, the correct version for coherent 2-groups will need to be given in the ``up to isomorphism" sense.  The following is inspired by the discussion at \cite{StabilizerDiscussionOnline}.  The stabilizer of an object $X\in\X$ under the action $M\colon \X\times \G\to\X$ of a coherent 2-group will be given by the pseudo-pullback
$$\begin{tikzpicture}
\node(1){$\stab(X)$};
\node(2)[node distance=1in, right of=1]{$\mathbf 1$};
\node(3)[node distance=.7in, below of=1]{$\G$};
\node(4)[node distance=.7in, below of=2]{$\X$};
\node(5)[node distance=.5in, right of=1]{$$};
\node(6)[node distance=.35in, below of=5]{$\cong$};
\draw[->](1) to node [above]{$$}(2);
\draw[->](1) to node [left]{$$}(3);
\draw[->](2) to node [right]{$X$}(4);
\draw[->](3) to node [below]{$X\cdot -$}(4);
\end{tikzpicture}$$
taken in $\mathfrak{Cat}$.  Since $X\colon \mathbf 1\to \X$ is faithful and faithful functors are stable under pseudo-pullback, the induced functor from the stabilizer to $\G$ is faithful.  Additionally, the images of the underlying sets of objects and of morphisms in the stabilizer are open in $\G_0$ and $\G_1$, respectively.  Thus, $\stab(X)$ will be an open sub-2-group in a sense modifying slightly that of Definition \ref{top sub 2-group}.
\end{remark}

\begin{example}[Fixed Points of Action of Open Sub-2-Group] \label{fixed points} Suppose that $\G$ acts continuously on the right of $\X$ viewed as a discrete 2-space.  Let $\U\subset\G$ denote an open sub-2-group as in Definition \ref{top sub 2-group}.  Let $\X^{\U}$ denote the subcategory of $\X$ having objects those $X\in\X_0$ such that $X\cdot U=X$ for all $U\in \X_0$ and $1_X\cdot g = 1_X$ for all $g\in\G_1$ and with arrows those $f\colon X\to Y$ of $\X$ such that $f\cdot 1_U=f$ for all $U\in \U_0$.
\end{example}

\begin{lemma}  Suppose that $\G$ acts continuously on $\X$.  The correspondence
\[ \mathscr U \mapsto \X^{\U} 
\]
extends to a functor $\mathbf L(\G)^{op} \to\mathbf{Cat}$.
\end{lemma}
\begin{proof}  The assignment on arrows is given by restriction since whatever is fixed by the action of an open sub-2-group is certainly fixed by any open sub-2-group contained it it.  \end{proof}

\section{The Site for Continuous Group Actions}

Here we review Elmendorf's Theorem in \cite{Elmendorf} showing that continous group actions can be organized into a Grothendieck topos.  The reference used for sheaf theory and toposes is Chapter III of \cite{MM}.

\subsection{Grothendieck Topologies and Sheaves}

A sieve on an object $C\in\C_0$ is a set of arrows $S$ with codomain $C$ satisfying the closure condition that if $f\in S$ and $fg$ is defined, then $fg\in S$ holds too.  Equivalently, a sieve $S$ on $C\in\C_0$ is a subfunctor of the canonical representable functor $\mathbf yC\colon \C^{op}\to\mathbf{Set}$.

\begin{define} \label{groth top} A \textbf{Grothendieck topology} $J$ assigns to each object $U\in \C$ a set of sieves $J(U)$ on $U$ in such a way that
\begin{enumerate}
\item the total sieve $\lbrace f\in \C_1\mid d_1f=U\rbrace$ on $U$ is in $J(U)$;
\item if $S$ is in $J(U)$ and $h\colon V\to W$ is any arrow, then the pullback $h^*S$ sieve is in $J(d_0 h)$;
\item if $S$ is in $J(U)$ and $R$ is any sieve on $U$ such that $h^*R\in J(W)$ holds for any arrow $h\colon V\to U$ of $S$, then $R$ is in $J(U)$ too.
\end{enumerate}
The objects of $J(U)$ are called ``covering sieves" or just ``coverings" or even ``covers."  A site is a pair $(\C,J)$ for a small category $\C$ and a topology $J$.
\end{define}

Every set of arrows with codomain $C\in \C_0$ generates a sieve on $C$.  Grothendieck topologies are often defined by giving certain distinguished ``basis subsets" of arrows whose generated sieves then generate a topology under appropriate conditions.

\begin{example} \label{atomic top} The atomic topology is that given by 
\[ JC = \lbrace S\mid S\;\text{is a nonempty sieve on}\; C\rbrace.
\]
For this definition to satisfy the axioms of Definition \ref{groth top}, any cospan of $\C$ must admit a span making a commutative square as indicated by the dashed arrows in 
$$\begin{tikzpicture}
\node(1){$\cdot$};
\node(2)[node distance=1in, right of=1]{$\cdot$};
\node(3)[node distance=.7in, below of=1]{$\cdot$};
\node(4)[node distance=.7in, below of=2]{$\cdot$};
\node(5)[node distance=.5in, right of=1]{$$};
\node(6)[node distance=.35in, below of=5]{$=$};
\draw[->,dashed](1) to node [above]{$$}(2);
\draw[->,dashed](1) to node [left]{$$}(3);
\draw[->](2) to node [right]{$$}(4);
\draw[->](3) to node [below]{$$}(4);
\end{tikzpicture}$$
This is called the ``Right Ore Condition" in \cite{Elephant} and \cite{Elephant2}.  Evidently this condition holds, for example, if $\C$ has pullbacks.  The atomic topology has as a basis the sieves generated by singletons.
\end{example}

\begin{define} Let $F\colon \C^{op}\to\mathbf{Set}$ denote a presheaf.  A \textbf{matching family} for $F$ on a cover $S$ of $C\in\C_0$ is a natural transformation $\theta\colon S\to F$ viewing $S$ as a subfunctor of $\mathbf yC$.  Denote the image of the arrow $f\colon B\to C$ in $S$ under $\theta_C\colon SC\to FC$ by $x_f$.
\end{define}

\begin{define} \label{sheaf def} A presheaf $F\colon \C^{op}\to\mathbf{Set}$ is a \textbf{sheaf} for $J$, or a $J$-sheaf, if every matching family on any object $C\in \C$ has a unique amalgamation, namely, a unique element $x\in FC$ such that $Ff(x) = x_f$ for each covering arrow $f\colon B\to C$.  A Grothendieck topos is a category $\E$ equivalent to a category of sheaves on a site $\E\simeq \mathbf{Sh}(\C,J)$.
\end{define}

\begin{remark}  The statment of the sheaf condition in Definition \ref{sheaf def} is just that pullback by $m\colon S\to \mathbf yC$, a covering sieve, induces an bijection of hom-sets
\begin{equation} \label{sheaf isomorphism} [\C^{op},\mathbf{Set}](\mathbf yC,F)\cong [\C^{op},\mathbf{Set}](S,F)
\end{equation}
where $[\C^{op},\mathbf{Set}]$ denotes the 1-category of presheaves on $\C$.
\end{remark}

\begin{lemma} \label{sheaf atomic top characterization} A presheaf $F\colon \C^{op}\to\mathbf{Set}$ is a sheaf for the atomic topology on $\C$ if, and only if, for any $x\in FC$ and any arrow $f\colon C\to D$ such that $Fh(x) = Fk(x)$ holds for all diagrams
$$\begin{tikzpicture}
\node(1){$B$};
\node(2)[node distance=.8in, right of=1]{$C$};
\node(3)[node distance=.8in, right of=2]{$D$};
\draw[->,transform canvas={yshift=.8ex}](1) to node [above]{$h$}(2);
\draw[->,transform canvas={yshift=-0.8ex}](1) to node [below]{$k$}(2);
\draw[->](2) to node [above]{$f$}(3);
\end{tikzpicture}$$
satisfying $fh=fk$, there is a unique $y\in FD$ such that $Ff(y)=x$.
\end{lemma}
\begin{proof} See III.4.2 of \cite{MM}. \end{proof}

\begin{remark}  The condition of Lemma \ref{sheaf atomic top characterization} is a well-definition condition as is seen in the proof in the reference.  This will reappear in dimension 2 in the proofs of Propositions \ref{surjective when sheaf} and \ref{full when sheaf}.  \end{remark}

\subsection{Review of the Site for Continuous Group Actions}

A textbook account of Elmendorf's result with some simplifications appears, for example, in \S III.9 of \cite{MM}.  The development here points out a few subtleties for the 2-group generalization.

Throughout let $G$ denote a topological group with $e\in \G$ its unit element; and $\mathbf BG$, the category of continuous right $G$-actions on discrete topological spaces $X$ and $G$-equivariant maps between them.  Call the objects of $\mathbf BG$ ``continuous (right) $G$-sets."

For such a $G$-set $X$ and an open subgroup $U\subset G$, let $X^U$ denote the set of $U$-fixed points, namely,
\begin{equation} \label{U-fixed points}
X^U=\lbrace x\in X\mid xg=x\; \text{for all}\;g\in U\rbrace.
\end{equation}
The right cosets of $U$ are sets of the form
\[ Ug= \lbrace ug\mid u\in U\rbrace
\]
for $g\in G$.  Cosets $Ux$ and $Uy$ are equal if, and only if, $xy^{-1}\in U$ holds.  Now, the ``orbit space of right cosets" is $G/U = \lbrace Ug\mid g\in G\rbrace$.  This is characterized in the following way.  

\begin{lemma}  The orbit of cosets $G/U$ as above is the coequalizer
$$\begin{tikzpicture}
\pgfmathsetmacro{\shift}{0.3ex}
\node(1){$U\times G$};
\node(2)[node distance=1in, right of=1]{$G$};
\node(3)[node distance=.8in, right of=2]{$G/U$};
\draw[->,dashed](2) to node [above]{$q$}(3);
\draw[transform canvas={yshift=0.5ex},->](1) to node [above]{$\pi$}(2);
\draw[transform canvas={yshift=-0.5ex},->](1) to node [below]{$-\otimes-$}(2);
\end{tikzpicture}$$
taken in $\mathbf{Set}$ where the group law is restricted to $U$ via the inclusion $U\subset G$.
\end{lemma}

\begin{remark} \label{action on orbit space} The quotient topology on $G/U$ turns out to be discrete since $U$ is open in $G$.  Each orbit space of cosets admits a continuous right action of $G$, making it an object of $\mathbf BG$.  The action is given by right multiplication on representatives $(Ux)\cdot g = Uxg$.
\end{remark}

Now it is appropriate to introduce the underlying category of the site for $\mathbf BG$.  Let $\mathbf S(G)$ denote the following category.  Objects are coset spaces $G/U$ for $U\subset G$ open subgroups.  Arrows $G/U\to G/V$ are objects $g\in G$ for which the inclusion $U\subset g^{-1}Vg$ holds.  This inclusion is equivalent to the statement that the map $G/U\to G/V$ given by $Ux\mapsto Vgx$ is well-defined.  Identities are thus represented by the neutral element $e\in G$.  Composition is $G$-multiplication.  

Equip $\mathbf S(G)$ with the atomic topology as in Example \ref{atomic top}.  The main result is the following.

\begin{theo}[Elmendorf \cite{Elmendorf}]  There is an equivalence of categories
\[ \mathbf BG\simeq \mathbf{Sh}(\mathbf S(G),J_{at})
\]
where $J_{at}$ denotes the atomic topology on $\mathbf S(G)$.
\end{theo}

The full proof followed here appears in \S III.9 of \cite{MM}.  However, some comments are in order to guide the categorification in \S 4.  The functor
\[ \Phi\colon\mathbf BG \to [\mathbf S(G)^{op},\mathbf{Set}]
\]
takes a continuous $G$-set $X$ to the representable presheaf $\mathbf BG(-,X)\colon \mathbf S(G)^{op}\to\mathbf{Set}$.  On the other hand, the functor 
\[ \Psi\colon [\mathbf S(G)^{op},\mathbf{Set}] \to \mathbf BG
\]
takes a presheaf $F$ on $\mathbf S(G)$ to the colimit
\[ \lim_{\mathclap{\substack{\to \\ U }}} F(G/U)
\]
taken in $\mathbf{Set}$ over all open subgroups $U\subset G$.  This is well-defined because the colimit admits a continuous right action of $G$.  For the colimit is filtered, hence computable (as in \S 2.13 of \cite{Handbook1}) as a set of germ-like equivalence classes $[U,x]$ for $x\in F(G/U)$ under the relation $(U,x)\sim (V,y)$ if, and only if, there is an open subgroup $W\subset U\cap V$ such that $x\cdot e =y\cdot e$ holds in $F(G/W)$.  The action is then given by declaring $[U,x]\cdot g = [g^{-1}Ug,x\cdot g]$.

Now, let us consider an outline of the equivalence. For each continuous $G$-set $X$, there is always a $G$-equivariant isomorphism
\begin{equation} \label{equiv iso 1} X\cong \lim_{\to} \mathbf BG(-,X)
\end{equation}
since $\mathbf BG(-,X)$ can be computed as the set of $U$-fixed points, namely, as
\begin{equation} \label{fixed point isomorphis} \mathbf BG(G/U,X) \cong X^U
\end{equation}
naturally in $U$.  The isomorphism \ref{equiv iso 1} is natural in $X$ and thus yields one half of the equivalence.  On the other hand, if $F$ is a sheaf for the atomic topology, then there are natural isomorphisms
\begin{equation} \label{equiv sheaf iso} F(G/U) \cong \mathbf BG(G/U,\lim_{\mathclap{\substack{\to \\ V}}} F(G/V)).
\end{equation}
This follows first by using the isomorphism in Equation \ref{fixed point isomorphis}.  From the colimit computation, the induced map can be seen to be always injective, and to be surjective whenever $F$ is a sheaf by the criterion of Lemma \ref{sheaf atomic top characterization}.  Thus, restricting to the presheaves that are sheaves, the desired equivalence is obtained.

\section{The 2-Site for Continuous 2-Group Actions}

This section gives three preliminary subsections of definitions and results necessary to prove the main theorem of the paper, namely, Theorem \ref{MAIN THEOREM}.  First in \S 4.1 is reviewed the appropriate notion of ``2-dimensional sheaf."  Following this in \S 4.2 is the construction of the underlying 2-category of the 2-site for $\mathfrak B\G$.  And \S 4.3 includes the necessary technical preliminaries for the main theorem.

\subsection{2-Sites and 2-Sheaves}

Topologies and sites in dimension 2 were probably first considered in Street's papers \cite{StreetSheafThy} and \cite{StreetCharacterization}.  Here we nearly follow the presentation of the former, as it gives the correct setting for the strict 2-groups considered here.  Development for coherent 2-groups will be based on the account of 2- and bisites in terms of ``coverages" as in \S C2 of \cite{Elephant2} and their ``up to isomorphism" generalization at \cite{nlab:2-sheaf}.

\begin{define}[Cf. \S 3.1 \cite{StreetSheafThy}] \label{2-sieve defn}  A \textbf{sieve} on an object $C\in\mathfrak C_0$ is a pointwise injective-on-objects and faithful 2-natural transformation $S\to\mathbf yC$.  A \textbf{topology} on a 2-category $\mathfrak C$ is a collection of such sieves that constitute a Grothendieck topology, as in Definition \ref{groth top}, on 1-category $|\mathfrak C|$ underlying $\mathfrak C$.  A 2-site is a 2-category $\mathfrak C$ together with a topology $J$, displayed as $(\mathfrak C,J)$.
\end{define}

\begin{remark} \label{Remark on Street} In the reference, Street defines a sieve on such an object $C$ to be an chronic arrow $S\to \mathbf yC$ in $[\mathfrak C^{op},\mathfrak{Cat}]$, that is, a pointwise injective-on-objects and fully faithful 2-natural transformation $S\to \mathbf yC$.  This definition, however, does not work for the present example of the site for continuous 2-group actions.  For requiring ``full" as a part of the definition would rule out the following example as a topology.  Additionally, the requirement would cause the proof of the technical result Proposition \ref{surjective when sheaf} to fail.
\end{remark}

\begin{example} \label{atomic 2-topology} The \textbf{atomic topology} on a 2-category $\mathfrak C$ is given by taking as basis families on $C\in \mathfrak C_0$ the singletons $\lbrace f\colon D\to C\rbrace$ over all arrows $f\colon D\to C$ with codomain $C$.  For this it must be supposed that each cospan completes to a commutative square as in Example \ref{atomic top}.  Thus, in this case, covering sieves are those generated by singletons.  That is, if $C\in\mathfrak C_0$ is an object, the covering sieve generated by $f\colon D\to C$ consists of all arrows $g\colon B\to C$ factoring through $f$ by some $k\colon B\to D$.  A morphism of such arrows $g$ and $h$ is a 2-cell $\alpha\colon g\Rightarrow h$ that factors through $f$ via another 2-cell say $\beta\colon k\to l$ in the sense that $f\ast \beta=\alpha$ holds where $k$ and $l$ factor $g$ and $h$ through $f$, respectively.  A 2-category equipped with the atomic topology is called an atomic 2-site.  Notice that requiring ``full" would not work here because not every 2-cell with codomain $C$ will factor through $f$ unless $f$ is representably full.  
\end{example}

The notion of ``2-sheaf" used here is essentially that of \S 3.1 of \cite{StreetSheafThy}.  It is a ``2-ified," or more specifically ``enriched," version of \ref{sheaf isomorphism}.

\begin{define}[Cf. \S 3.1 \cite{StreetSheafThy}] \label{defn 2-sheaf}  A \textbf{2-sheaf} for 2-site $(\mathfrak C,J)$ is a 2-functor $F\colon \mathfrak C^{op}\to\mathfrak{Cat}$ such that for any covering sieve $S\to \mathbf yC$ on any object $C\in \mathfrak C$, there is induced an isomorphism of categories
\begin{equation} \label{2-sheaf isomor} [\mathfrak C^{op},\mathfrak{Cat}](\mathbf yC,F)\cong [\mathfrak C^{op},\mathfrak{Cat}](S, F)
\end{equation}
given by pulling back along $S\to \mathbf yC$.  Let $\mathfrak{Sh}(\mathfrak C,J)$ denote the 2-category of 2-sheaves on the 2-site $(\mathfrak C,J)$ together with 2-natural transformations and their modifications.
\end{define}

Think of the 2-natural transformations on the right side of \ref{2-sheaf isomor} as ``compatible data on the cover $S$."  Note that if $(\mathfrak C,J)$ is any site, then each $X\in FC$ for $C\in\mathfrak C$ determines compatible object data on any covering sieve $S$ on $C$.  This is given by taking $X_f$ to be $Ff(X)$ for each $f\in S$.  Similarly for arrows of $FC$.  A crucial technical result to be used later is the following.

\begin{lemma}\label{atomic top lemma}  Let $(\mathfrak C,J)$ denote an atomic 2-site and $F$ a 2-sheaf.  For each $f\colon C\to D$, the induced functor $Ff \colon FD\to FC$ is injective-on-objects and faithful.
\end{lemma}
\begin{proof}  Any two $X,Y\in F(D)_0$ for which $X\cdot f = Y\cdot f$ holds determine the same data on the covering sieve generated by $\lbrace f\colon C\to D \rbrace$.  Thus, $X=Y$ holds.  A similar argument shows that $Ff$ is faithful.  \end{proof}

\subsection{The Underlying 2-Category}

Let $\G$ denote a topological 2-group as in Definition \ref{top 2-group}.  The goal is to define the underlying 2-category $\mathfrak S(\G)$ of a proposed 2-site for $\mathfrak B\G$.

\begin{lemma}  Let $\mathscr U\subset\G$ denote an open sub-2-group of $\G$ as in Definition \ref{top sub 2-group}.  The groups $A^{-1}\U_0A$ and $1_{A^{-1}}\U_11_A$ with operations inherited from $\G$ yield an open sub-2-group denoted simply by $A^{-1}\U A$.
\end{lemma}
\begin{proof}  The conjugate subgroups $A^{-1}\U_0A$ and $1_{A^{-1}}\U_11_A$ are open in $\G_0$ and in $\G_1$, respectively.  The subgroup and 2-group structure of $A^{-1}\U A$ is inherited from that of $\U$.  \end{proof}

Define $\G/\mathscr U$ to be the category whose objects are cosets $\U_0A$ for $A\in \G_0$ and whose arrows are cosets $\U_1g$ for $g\in \G_1$.  In other words, declare
\[ (\G/\U)_0 := \G_0/\U_0 \qquad (\G/\U)_1 := \G_1/\U_1.
\]
Thus, if $g\colon A\to B$ denotes an arrow of $\G$, then the domain of $\U_1g$ is $\U_0A$ while the codomain is $\U_0B$.  These assignments are well-defined, for if $h\colon C\to D$ has $\U_1g=\U_1h$, then $gh^{-1}\in \U_1$, so that $A\otimes C^{-1}, B\otimes D^{-1}\in\U_0$, since $\U\subset \G$ is a subcategory.  The identity on $\U_0A$ is $\U_11_A$, which is well-defined again since $\U$ is a subcategory.  Composition is defined to be the coset of the $\G$-composition of arbitrary representatives.  This is well-defined by the interchange law \ref{interchange} in $\G$. 

Now, $\G$ acts continuously on the right of each coset category $\G/\U$.  For this, consider the ordinary actions on the underlying sets given by 
\begin{equation} \label{action on 2-orbit obj} \G_0/\U_0\times \G_0 \to \G_0/\U_0 \qquad (\U_0A,B)\mapsto \U_0(A\otimes B)
\end{equation}
and 
\begin{equation} \label{action on 2-orbit arr} \G_1/\U_1\times \G_1 \to \G_1/\U_1 \qquad (\U_1g,h)\mapsto \U_1(g\otimes h).
\end{equation}
Each is a continuous action of a group on a set.  So, it remains only to see that these functions are the underlying functions of a functor $\G/\U\times \G\to\G$.  The domain, codomain and identity laws hold by the construction of $\G/\U$ as a category.  That composition is preserved follows from the interchange law in $\G$.

\begin{prop}  If $\U$ is an open sub-2-group of $\G$, then $\G$ acts continuously on the right of $\G/\U$ in the sense of Definition \ref{action}.  Additionally, $\G/\U$ is an object of $\mathfrak B\G$.
\end{prop}
\begin{proof}  The first statement was proved in the discussion above.  The topologies on each of the sets $\G_0/\U_0$ and $\G_1/\U_1$ are discrete because $\U_0\subset \G_0$ and $\U_1\subset \G_1$ are open subgroups.  \end{proof}

\begin{construction} \label{underlying 2-cat construction} Now, define the underlying 2-category.  Objects of $\mathfrak S(\G)$ will be those categories $\G/\U$ over the open sub-2-groups $\mathscr U\subset \G$.  A morphism $\G/\U\to \G/\V$ will be an object $A\in \G_0$ such that the containments
\begin{equation} \label{well-defn cond1} \U_0\subset A^{-1}\V_0A \qquad \U_1\subset 1_{A^{-1}} \V_11_A
\end{equation} 
both hold, i.e. that $\U\subset A^{-1}\V A$ holds as open sub-2-groups of $\G$.  The containment is equivalent to requiring that the assignments between ordinary coset spaces 
\begin{equation} \label{well-def mor 2-orbit cat1}\G_0/\U_0 \to \G_0/\V_0, \qquad \U_0X \mapsto \V_0(A\otimes X )
\end{equation}
\begin{equation}\label{well-def mor 2-orbit cat2}
\G_1/\U_1\to\G_1/\V_1,\qquad \U_1g\mapsto \V_1(1_A\otimes g)
\end{equation}
are well-defined.  Each $A\colon \G/\U\to\G/\V$ satisfying the well-definition condition $\U\subset A^{-1}\V A$ is indeed a functor.  Accordingly, each $g\colon A\to B$ of $\G$ inducing a 2-cell $g\colon A\Rightarrow B$ of $\mathfrak S(\G)$ ought to be a natural transformation.  The definition of the $\U_0X$-component will be
\begin{equation} \label{2-cell cond for 2-orbit cat} g_{\U_0X}:=\V_1g\otimes 1_X.
\end{equation}
Provided this is well-defined, naturality follows by the bifunctor condition for $\otimes \colon \G\times\G\to\G$.  Well-definition is equivalent to the statement 
\begin{equation} \label{well-definition 2-cells orbit cat} (X\in\U_0)\;\; \text{implies that}\;\; (g\otimes 1_X\otimes g^{-1}\in \V_1).
\end{equation}
Thus, take the 2-cells of $\mathfrak S(\G)$ to be those $g\colon A\Rightarrow B$ given by \ref{2-cell cond for 2-orbit cat} and satisfying the well-definition condition \ref{well-definition 2-cells orbit cat}.  Composition of morphisms $A\colon \G/\U\to\G/\V$ and $B\colon \G/\V\to\G/\W$ is defined to be $B\otimes A\colon \G/\U\to\G/\W$.  Horizontal composition of 2-cells in $\mathfrak S(\G)$ is defined to be the tensor in $\G$.  Well-definition follows by the bifunctor condition and the assumed well-definition conditions.  Vertical composition is defined using the composition in $\G$.  Well-definition follows again by the bifunctor condition and the fact that $\V$ is a subcategory.    
\end{construction}

\begin{lemma}  The data and operations of the discussion above make $\mathfrak S(\G)$ a 2-category.
\end{lemma}
\begin{proof}  Associativity and unit laws of each composition, as well as the interchange law relating the two 2-cell compositions, are all inherited from the corresponding properties of $\G$.  Notice that if $\G$ is a weak monoidal category, then $\mathfrak S(\G)$ will be a bicategory and not a (strict) 2-category.  \end{proof}

Note that each morphism $A\colon \G/\U\to\G/\V$ of $\mathfrak S(\G)$ is $\G$-equivariant and that each $g\colon A\Rightarrow B$ satisfies the 2-cell compatibility condition \ref{action 2-cell compat}.  Thus, each morphism of $\mathfrak S(\G)$ is a morphism of $\mathfrak B\G$ and each 2-cell of $\mathfrak S(\G)$ is one of $\mathfrak B\G$.

Now, by construction, every morphism of $\mathfrak S(\G)$ is basically a pair of surjective set functions.  For this reason, thinking of surjectivity as ``covering," declare those sieves on objects of $\mathfrak S(\G)$ to be covering that are generated by singletons.  This is a topology on $\mathfrak S(\G)$ because the Right Ore Condition of Example \ref{atomic top} is satisfied in a manner similar to the ordinary 1-category $S(G)$ for a topological group $G$. 

\subsection{Technical Results}

Throughout $\mathbf L(\G)$ denotes the poset of open sub-2-groups of $\G$, viewed as a locally discrete 2-category.  Let $F\colon \mathfrak S(\G)^{op}\to\mathfrak{Cat}$ denote a 2-functor with $\mathfrak S(\G)$ as in Construction \ref{underlying 2-cat construction}.

\begin{construction}[Colimit Category]  Construct what will be a category 
\begin{equation} \label{colimit cat} \lim_{\mathclap{\substack{\to \\\U }}} F(\G/\U)
\end{equation}
in the following way.  Take as a set of objects
\begin{equation} \label{objects}
(\lim_{\mathclap{\substack{\to \\\U }}} F(\G/\U))_0 := \lim_{\mathclap{\substack{\to \\ \U}}} F(\G/\U)_0
\end{equation}
and as a set of arrows
\begin{equation} \label{arrows}
(\lim_{\mathclap{\substack{\to \\\U }}} F(\G/\U))_1 := \lim_{\mathclap{\substack{\to \\ \U}}} F(\G/\U)_1
\end{equation}
with each colimit taken over $\mathbf L(\G)$ in the category of sets.  These are filtered colimits.  Thus, by the computation of \S 2.13 of \cite{Handbook1}, these sets are described in terms of certain equivalence relations.  Objects are equivalence classes of pairs $[\U,X]$ with $X\in F(\G/\U)_0$ under the relation
\begin{equation} (\U,X)\sim (\V,Y) \qquad \text{if, and only if} \qquad \exists\; \mathscr W\subset \U\cup \V\;\text{such that} \; X\cdot I = Y\cdot I.
\end{equation}
Similarly, arrows are equivalence classes $[\U,m]$ with $m\in F(\G/\U)_1$ under the relation
\begin{equation} (\U,m)\sim (\V, n) \qquad \text{if, and only if} \qquad \exists\; \mathscr W\subset \U\cup \V\;\text{such that} \; m\cdot I = n\cdot I.
\end{equation}
Take the domain of an arrow $[\U,m]$ with $m\colon X\to Y$ to be $[\U, X]$ and the codomain to be $[\U,Y]$.  The identity on $[\U,X]$ is taken to be $[\U,1_X]$.  Composition of $[\U,n]$ and $[\V,m]$ with $[\U, \dom n]=[\V,\cod m]$ is taken to be $[\mathscr W,(n\cdot I)(m\cdot I)]$ for any $\mathscr W\subset \U\cap \V$.  This is independent of $\mathscr W$ and the choice of representatives.
\end{construction}

\begin{prop} \label{colimit prop} If $F\colon \mathfrak S(\G)^{op}\to\mathfrak{Cat}$ is a 2-functor, then, with objects and arrows as in \ref{objects} and \ref{arrows} and with operations as described in the discussion above,
\[ \lim_{\mathclap{\substack{\to \\\U }}} F(\G/\U)
\]
is a category.  Additionally, it is the 2-colimit in the sense of \ref{2-colimit} of the induced 2-functor $F\colon \mathbf L(\G)^{op} \to\mathfrak{Cat}$ given by restricting $F$ to the poset of open sub-2-groups.
\end{prop}
\begin{proof}  The construction of the filtered colimit of any functor valued in $\mathbf{Cat}$ is discussed in Example 5.2.2.f of \cite{Handbook2}.  The present construction of the category in \ref{colimit cat} is a special case.  However, this is indeed the 2-colimit as well.  For it is routine to construct functors 
\[ \Phi\colon \mathfrak{Cat}(\lim_{\mathclap{\substack{\to \\ \U}}}F(\G/\U),\A)\rightleftarrows [\mathbf L(\G)^{op},\mathfrak{Cat}](\mathbf 1,\mathfrak{Cat}(F,\A))\colon \Psi
\]
and show that these are mutually inverse to give the isomorphism of \ref{2-colimit}.  \end{proof}

The next goal is to show that the colimit category of \ref{colimit cat} is an object of $\mathfrak B\G$.  First define correspondences at the object and arrow levels.  Define an action of $\G_0$ on objects of the colimit by 
\begin{equation} \label{action obj assign} [\U,X]\cdot A := [A^{-1}\mathscr U A, X\cdot A]
\end{equation}
for $A\in \G_0$.  This is a well-defined, continuous action of $\G_0$ by the group laws and the fact that $\mathbf L(\G)$ is filtered.  The action at the level of arrows is more complicated.  For this, take an arrow $[\U,m]$ of the colimit represented by $m\colon X\to Y$ and take an arrow $g\colon A\to B$ of $\G$.  Let $\mathscr O$ denote the intersection
\[ \mathscr O := A^{-1}\U A\cap B^{-1}\U B.
\] 
By the construction of the 2-category $\mathfrak S(\G)$, the arrow $g$ induces a 2-cell 
$$\begin{tikzpicture}
\node(1){$\G/\mathscr O$};
\node(2)[node distance=1in, right of=1]{$\Downarrow g$};
\node(3)[node distance=1in, right of=2]{$\G/\U$};
\node(4)[node distance=.4in, above of=2]{$\G/A^{-1}\U A$};
\node(5)[node distance=.4in, below of=2]{$\G/B^{-1}\U B$};
\draw[->](1) to node [above]{$I$}(4);
\draw[->](1) to node [below]{$I$}(5);
\draw[->](4) to node [above]{$A$}(3);
\draw[->](5) to node [below]{$B$}(3);
\end{tikzpicture}$$
hence a natural transformation $Fg\colon FA \Rightarrow FB$.  For simplicity, denote the component of $Fg$ at an object $Z\in F(\G/\mathscr U)$ by $g_Z\colon Z\cdot A \to Z\cdot B$.  Now, define
\begin{equation} \label{action arrow assign} [\U,m]\cdot g:= [\mathscr O,\xi]
\end{equation}
where $\xi$ is either side of the commutative square
$$\begin{tikzpicture}
\node(1){$X\cdot A$};
\node(2)[node distance=1in, right of=1]{$Y\cdot A$};
\node(3)[node distance=.8in, below of=1]{$X\cdot B$};
\node(4)[node distance=.8in, below of=2]{$Y\cdot B$};
\node(5)[node distance=.5in, right of=1]{$$};
\node(6)[node distance=.4in, below of=5]{$=$};
\draw[->](1) to node [above]{$m\cdot A$}(2);
\draw[->](1) to node [left]{$g_X$}(3);
\draw[->](2) to node [right]{$g_Y$}(4);
\draw[->](3) to node [below]{$m\cdot B$}(4);
\end{tikzpicture}$$
of $F(\G/\mathscr O)$.  The assignment of \ref{action arrow assign} is well-defined by the 2-functoriality of $F$.

\begin{lemma} \label{action on colimit} The assignments of \ref{action obj assign} and \ref{action arrow assign} give a continuous action of $\G_1$ on the arrows of the category \[\lim_{\mathclap{\substack{\to \\\U }}} F(\G/\U)
\]
compatible with the action of $\G_0$ on the objects in the sense that the two together yield a functor 
\[ \lim_{\mathclap{\substack{\to \\\U }}} F(\G/\U)\times \G\to \lim_{\mathclap{\substack{\to \\\U }}} F(\G/\U)
\]
satisfying the properties of a continuous action of $\G$ on the colimit.
\end{lemma}
\begin{proof}  That the assignments constitute a functor satisfying the axioms for an action is a tedious but ultimately straightforward check using the definition of the actions and the relation forming the colimit.  Continuity of the action at the arrow-level requires some discussion, however.  It suffices to see that the inverse image of a point $[\U,m]$ is an open subset of the domain.  Let $g\colon A\to B$ be any arrow of $\G$.  Consider the class $[\mathscr O, \xi]$ with $\mathscr O$ and $\xi$ as in the discussion above.  By the properties of $F$ as a 2-functor, it follows that 
\[ [\mathscr O, \xi]\cdot g = [\U,m]
\]
holds, meaning that the inverse image of the action is some subset of the arrows of the colimit crossed by direct product with $\G_1$ itself.  Since the arrows of the colimit have the discrete topology, this is an open subset of the domain.  Thus, the action is continuous.\end{proof}

Now, let $\X$ denote a discrete 2-space on which $\G$ acts continuously on the right.  Let $\mathscr U\subset \G$ denote an open sub-2-group.  The claim is that there is an isomorphism of categories
\begin{equation} \label{fixed pts isom} \mathfrak B\G(\G/\mathscr U,\X)\cong \X^{\mathscr U}
\end{equation}
that is, the 2-dimensional analogue of \ref{fixed point isomorphis} for the fixed point open sub-2-group of Example \ref{U-fixed points}.  On the one hand, given a $\G$-equivariant functor $F\colon \G/\mathscr U\to \X$ as on the left side of \ref{fixed pts isom}, define $\Phi(F)$ to be $F(\mathscr UI)$.  This is well-defined by $\G$-equivariance.  On an arrow $\theta\colon F\to G$ on the left side of \ref{fixed pts isom}, define $\Phi(\theta)$ to be the component 
\[ \theta_{\mathscr UI}\colon F(\mathscr UI)\to G(\mathscr UI).
\] 
This is well-defined by the compatibility condition assumed for $\theta$.  This $\Phi$ is a functor by the definition of composition of natural transformations and since $\G$ is a strict 2-group.  On the other hand, starting with $X\in \X_0$ fixed by $\mathscr U$, take $\Psi$ to be the functor $F_X\colon \G/\mathscr U\to \X$ given by 
\[ F_X(\mathscr U_0A) := X\cdot A 
\]
on objects and on arrows by
\[  F_X(\mathscr U_1g) := 1_X\cdot g.
\]
These assignments for $F_X$ are well-defined by the assumption that $X$ is $\mathscr U$-fixed.  Moreover, $F_X$ is a functor by the fact that the action of $\G$ on $\X$ is a bifunctor.  For an arrow $m\colon X\to Y$ of $\X$ fixed by $\mathscr U$, define what will be a transformation $m\colon F_X\to F_Y$ by taking the component at $\mathscr U_0A$ to be
\[ m_{\mathscr U_0A} = m\cdot 1_A.
\]
This is well-defined again because $m$ is assumed to be $\mathscr U$-fixed.  Naturality follows since the action of $\G$ is a bifunctor.  Finally, $\Psi$ is a functor by the bifunctor condition and the definition of composition of natural transformations.

Finally, the functors $\Phi$ and $\Psi$ are mutually inverse.  This is a straightforward computation from the definitions.  For example, at the object-level, consider
\[ \Psi\Phi(F) = \Psi(F(\mathscr U_0I)) = F_X
\]
with $X = F(\mathscr U_0I)$.  But compute that 
\[ F_X(\mathscr U_0A) = X\cdot A = F(\mathscr U_0I)A = F(\mathscr U_0I\otimes A) = F(\mathscr U_0A)
\]
by definition and equivariance of $F$.  A similar argument shows that $F$ and $F_X$ agree on arrows of $\mathscr G/\mathscr U$.  Thus, $\Psi\Phi$ is identity on objects.  The compatibility condition for transformations of equivariant functors shows that it is also identity on morphisms.  That is, if $\theta$ is a transformation of $\G$-equivariant functors $F,G\colon \G/\mathscr U\rightrightarrows \X$, then by the definition of the induced transformation and the compatibility condition, it follows that 
\[ (\theta_{\mathscr U_0I})_{\mathscr U_0A} = \theta_{\mathscr U_0I}\cdot 1_{\mathscr U_1A} = \theta_{\mathscr U_0I\otimes A} = \theta_{\mathscr U_0A}
\] 
as required.  On the other hand, the argument that $\Phi\Psi =1$ is similar but easier.  The result is now asserted formally in the following

\begin{prop} \label{fixed point iso prop} For a topological 2-group $\G$, open sub-2-group $\mathscr U\subset \G$, and a discrete 2-space $\X$ on which $\G$ acts on the right, there is an isomorphism of categories 
\[
\mathfrak B\G(\G/\mathscr U,\X)\cong \X^{\mathscr U}
\]
where $\mathfrak B\G$ denotes the 2-category of continuous right actions of $\G$ and $\X^{\mathscr U}$ denotes the category of $\mathscr U$-fixed points under the continuous action on $\X$.
\end{prop}

For each $\X$ in $\mathfrak B\G$, the assignment $\G/\U\mapsto \mathfrak B\G(\G/\U,\X)$ extends to a 2-functor $\mathfrak S(\G)^{op}\to\mathfrak{Cat}$.

\begin{cor} \label{fixed points is a functor on orbits} The assignment on objects 
\begin{equation} \label{colimit 2-functor} \G/\U \mapsto (\lim_{\mathclap{\substack{\to \\\V }}} F(\G/\V))^{\U}
\end{equation}
extends to a well-defined contravariant 2-functor on the orbit 2-category $\mathfrak S(\G)$.  Additionally, the fixed point isomorphism of Proposition \ref{fixed point iso prop}, namely,
\begin{equation} \label{naturality colimit iso} \mathfrak B(\G/\U,\lim_{\mathclap{\substack{\to \\\V }}} F(\G/\V)) \cong (\lim_{\mathclap{\substack{\to \\\V }}} F(\G/\V))^{\U}
\end{equation}
is a 2-natural isomorphism of 2-functors $\mathfrak S(\G)^{op}\to\mathfrak{Cat}$.  In other words, the colimit 2-functor given by \ref{colimit 2-functor} is representable.
\end{cor}
\begin{proof}  For a given morphism $A\colon \G/\U\to \G/\W$ in $\mathfrak S(\G)$, define the transition functor
\[ [-]\cdot A\colon  (\lim_{\mathclap{\substack{\to \\\V }}} F(\G/\V))^{\W} \to (\lim_{\mathclap{\substack{\to \\\V }}} F(\G/\V))^{\U}
\]
by taking $[\V,X]\mapsto [A^{-1}\V A,X\cdot A]$ on objects and analogously on arrows.  This is well-defined by the containment $\U\subset A^{-1}\V A$.  It is a functor by the 2-functor axioms for $F$ and the definition of composition in the colimit.  For a 2-cell $g\colon A\Rightarrow B$ in $\mathfrak S(\G)$, define what will be a natural transformation 
$$\begin{tikzpicture}
\node(1){$\displaystyle (\lim_{\mathclap{\substack{\to \\\V }}} F(\G/\V))^{\U}$};
\node(2)[node distance=1in, right of=1]{$\Downarrow [-]\cdot g$};
\node(3)[node distance=1in, right of=2]{$\displaystyle (\lim_{\mathclap{\substack{\to \\\V }}} F(\G/\V))^{\U}$};
\draw[->,bend left=60](1) to node [above]{$[-]\cdot A$}(3);
\draw[->,bend right=60](1) to node [below]{$[-]\cdot B$}(3);
\end{tikzpicture}$$ 
by taking the component at $[\V,X]$ to be the morphism
\[ [\mathscr O,g_X]\colon [A^{-1}\V A, X\cdot A]\longrightarrow [B^{-1}\V B,X\cdot B]
\]
where $\mathscr O = A^{-1}\V A\cap B^{-1}\V B$ and $g_X$ is the $X$-component of the transformation $Fg\colon FA\Rightarrow FB$ associated to $g$ under $F$.  This is natural by the functoriality of $F$ at the level of 2-cells.  The rest of the 2-functor axioms follow by the fact that $F$ is a 2-functor.

The isomorphism in \ref{naturality colimit iso} is 2-natural in $\G/\U$ by the definition of the transition morphisms and 2-cell assignment for the colimit 2-functor and $\G$-equivariance.  \end{proof}

Recall that $\mathbf L(\G)$ denotes the poset category of open sub-2-groups of $\G$.  The last technical result is an analogue of \ref{equiv iso 1}, namely, an isomorphism
\begin{equation}
\label{category is a colimi} \lim_{\mathclap{\substack{\to \\ \mathscr U}}}\X^{\mathscr U} \cong \X
\end{equation}
whenever $\X$ admits a continuous right action of $\G$.  On the one hand, given an object $[\mathscr U,X]$ in the colimit on the left of \ref{category is a colimi}, declare $\Phi[\mathscr U,X] =X$.  This is well-defined by the definition of $\mathscr X^{\mathscr U}$ as a functor on $\mathbf L(\G)$.  If $[\mathscr W,m]$ is a morphism of the colimit, take its image under $\Phi$ to be $m$.  Well-definition follows for the same reason.  Evidently $\Phi$ is a functor.  On the other hand, given $X$ in $\X$, take the image under $\Psi$ to be $[\stab(X),X]$.  Given a morphism $m\colon X\to Y$, take $\Phi(m)$ to be $[\stab(m),X]$.  This is well-defined by the definition of the equivalence classes.  Again $\Psi$ is a functor.

It is immediate that $\Phi\Psi = 1$ holds, as can be seen by just chasing objects and arrows through the definitions.  On the other hand, $\Psi\Phi=1$ holds essentially by the fact that if $X\in\X$ is $\mathscr U$-fixed, then $\mathscr U\subset \stab(X)$ as sub-2-groups, which implies that $[\stab(X),X]= [\mathscr U,X]$ holds too.  A similar argument works at the level of morphisms.  Each of the isomorphisms is natural in $\X$, as can be seen directly by computation from the definitions.  The result, then, is stated formally as

\begin{prop} \label{counit iso} Whenever $\X$ is a category admitting a continuous right action of a topological 2-group $\G$, there is an isomorphism of categories
\[\lim_{\mathclap{\substack{\to \\ \mathscr U}}}\X^{\mathscr U} \cong \X
\]
where the colimit on the left is taken over $\mathbf L(\G)$, the poset of open sub-2-groups of $\G$.  Additionally the isomorphism is natural in $\X$.
\end{prop}

\subsection{The 2-Equivalence}

A 2-functor
\begin{equation} \label{Big Phi} \Phi\colon \mathfrak B\G\to [\mathfrak S(\G)^{op},\mathfrak{Cat}]
\end{equation}
is given by taking $\X$ with continuous right action $M\colon \X\times \G\to\X$ to the representable 2-functor
\[ \mathfrak B\G(-,\X)\colon \mathfrak S(\G)^{op}\to\mathfrak{Cat}.
\]
The assignments on morphisms and 2-cells are given by pushing forward.  Well-definition of each assignment follows by the associativity of composition of functors and transformations.  On the other hand, a 2-functor
\begin{equation} \label{Big Psi} \Psi\colon [\mathfrak S(\G)^{op},\mathfrak{Cat}] \to \mathfrak B\G
\end{equation}
can be given by taking a 2-functor $F$ to the pseudo-colimit
\[\lim_{\mathclap{\substack{\to \\ \mathscr U}}} F(\G/\mathscr U) 
\]
taken over the poset $\mathbf L(\G)$ of open sub-2-groups $\mathscr U \subset \G$ as in Proposition \ref{colimit prop}.  Lemma \ref{action on colimit} shows that $\Psi$ is well-defined. If $\theta\colon F\Rightarrow G$ is a 2-natural transformation of 2-functors $\mathfrak S(\G)^{op}\to\mathfrak{Cat}$, then take $\Psi(\theta)$ to be the functor
\[ \Psi(\theta)\colon \lim_{\mathclap{\substack{\to \\ \U}}}F(\G/\U) \to \lim_{\mathclap{\substack{\to \\ \U}}}G(\G/\U)
\] 
given by
\[ \Psi(\theta)([\U,X]):= [\U,\theta_{\G/\U}(X)]
\] 
on objects and on morphisms by
\[ \Psi(\theta)([\U,m]): = [\U,\theta_{\G/\U}(m)].
\]
These assignments are well-defined by the 2-naturality of $\theta$.  And $\Psi(\theta)$ is $\G$-equivariant by the definition of the action of $\G$ on the colimits and 2-naturality of $\theta$.  Finally, the colimit definition of $\Psi$ on objects has enough structure to handle a 2-cell assignment as well.  For a modification $m\colon \theta \Rrightarrow \gamma$ of 2-natural transformations $\theta$ and $\gamma$, take $\Psi(m)$ to have components
\[ [\U,m_{\G/\U,X}] \colon [\U,\theta_{\G/\U}(X)] \to [\U,\gamma_{\G/\U}(X)].
\]
Indexed over all such $[\U,X]$, this definition comprises a natural transformation since each $m_{\G/\U}$ is natural.  The compatibility condition for 2-cells in $\mathfrak B(\G)$ follows by the modification condition for $m$.  From the definitions of the assignments and the definition of composition, it follows that $\Psi$ is a 2-functor.

The goal is to show that $\Phi$ and $\Psi$ of \ref{Big Phi} and \ref{Big Psi} are mutually pseudo-inverse when restricted to 2-sheaves.  This will exhibit the desired 2-equivalence yielding the main result of the paper.  First develop the counit.  This is always an isomorphism.

\begin{prop} In the notation of the discussion above, $1\cong \Psi\Phi$ holds 2-naturally in $\X\in \mathfrak B\G$. 
\end{prop}
\begin{proof} By combining Propositions \ref{fixed point iso prop} and \ref{counit iso}, already established is that there is such an isomorphism of categories for each $\X$ in $\mathfrak B\G$, namely,
\[  \Psi\Phi(\X) = \lim_{\mathclap{\substack{\to \\ \mathscr U}}}\mathfrak B\G(\G/\U,\X) \cong \lim_{\mathclap{\substack{\to \\ \mathscr U}}} \X^{\U}\cong \X.
\] 
But additionally this is a 2-natural isomorphism in $\mathfrak B\G$.  That it is 2-natural is easy to see since the value of $\Psi\Phi(\X)$ is just the colimit over a family of representable 2-functors.  The 2-cell condition for 2-naturality in particular follows by the definition of horizontal composition of natural transformations \ref{horiz comp in cat}.  The work consists in showing that the isomorphism above is $\G$-equivariant.  That is, to be established is that the diagram
$$\begin{tikzpicture}
\node(1){$\displaystyle \lim_{\mathclap{\substack{\to \\ \mathscr U}}}\mathfrak B\G(\G/\U,\X)\times \G$};
\node(2)[node distance=1.6in, right of=1]{$\X\times \G$};
\node(3)[node distance=.8in, below of=1]{$\displaystyle \lim_{\mathclap{\substack{\to \\ \mathscr U}}}\mathfrak B\G(\G/\U,\X)$};
\node(4)[node distance=.8in, below of=2]{$\X$};
\node(5)[node distance=.5in, right of=1]{$$};
\node(6)[node distance=.4in, below of=5]{$$};
\draw[->](1) to node [above]{$\cong$}(2);
\draw[->](1) to node [left]{$\mathrm{act}$}(3);
\draw[->](2) to node [right]{$\mathrm{act}$}(4);
\draw[->](3) to node [below]{$\cong$}(4);
\end{tikzpicture}$$
commutes.  At the level of objects, this follows because a representative $H\colon \G/\U\to \X$ of a class in the colimit is itself $\G$-equivariant.  On the other hand, let $(\theta\colon H\Rightarrow K, g\colon A\Rightarrow B)$ denote any arrow of the upper-left corner of the square.  Chasing it around each side of the square and comparing, it follows that we want to show there is an equality
\[  (\theta\ast g)_{\mathscr O_0I} = \theta_{\U_0I}\cdot g
\]
where the `$\ast$' denotes horizontal composition of natural transformations \ref{horiz comp in cat} and $\mathscr O$ denotes the intersection $A^{-1}\U A\cap B^{-1}\U B$ arising in the definition of the action of $\G$ on the colimit.  But that this equality holds is established by the following computation:
\begin{align}
(\theta\ast g)_{\mathscr O_0I} &= \theta_{\U_0B}\circ Hg_{\mathscr O_0I}  \qquad & \text{(def'n horiz comp \ref{horiz comp in cat})} \notag \\
&= \theta_{\U_0B}\circ H\U_1g  \qquad &\text{(def'n components of $g$ \ref{2-cell cond for 2-orbit cat})} \notag \\
&= (\theta_{\U_0I}\cdot 1_B)\circ H\U_1g  \qquad &\text{(compat for $\theta$ \ref{action 2-cell compat})} \notag \\
&= (\theta_{\U_0I}\cdot 1_B)\circ H(1_{\U_0I}\cdot g)  \qquad &\text{(def'n action of $g$ \ref{action on 2-orbit arr})} \notag \\
&= (\theta_{\U_0I}\cdot 1_B)\circ (1_{H\U_0I}\cdot g)  \qquad &\text{($H$ is $\G$-equivariant)} \notag \\
&= \theta_{\U_0I}\cdot g \qquad &\text{(action of $\G$ on $\X$ is a bifunctor)} \notag
\end{align}
Therefore, $1\cong \Psi\Phi$ holds 2-naturally in $\mathfrak B\G$, as asserted.  
\end{proof}

Now, we shall see that $F\cong \Phi\Psi(F)$ holds 2-naturally for 2-sheaves $F$.  First a preliminary result that constructs the unit.  In the statement, note that Proposition \ref{fixed point iso prop} has been used in calculating $\Phi\Psi(F)$.

\begin{lemma} \label{unit of 2-equivalence} For each 2-presheaf $F$ on $\mathfrak S(\G)$, there is a 2-natural transformation $\eta\colon F \Rightarrow \Phi\Psi(F)$ with components
\[ \eta \colon F(\G/\U) \longrightarrow  (\lim_{\mathclap{\substack{\to \\\V }}} F(\G/\V))^{\U}
\]
given on objects by sending $X\mapsto [\U,X]$.  These assemble into a 2-natural transformation $\eta\colon 1\Rightarrow \Phi\Psi$.
\end{lemma}
\begin{proof}  Notice that the arrow assignment $X\mapsto [\U,X]$ is well-defined since $[\U,X]$ is $\U$-fixed under the action of $\G$ on the colimit.  The arrow assignment sends $m\colon X\to Y$ to $[\U,m]$ and is similarly well-defined and functorial.  Now, to verify 2-naturality, let $A\colon \G/\U \to \G/\mathscr W$ denote a morphism of $\mathfrak S(\G)$.  Thus, in particular $\U\subset A^{-1}\mathscr WA$ holds as in \ref{well-def mor 2-orbit cat1} and \ref{well-def mor 2-orbit cat2}.  The square
$$\begin{tikzpicture}
\node(1){$F(\G/\mathscr W)$};
\node(2)[node distance=1.6in, right of=1]{$\displaystyle (\lim_{\mathclap{\substack{\to \\\V }}} F(\G/\V))^{\U}$};
\node(3)[node distance=.8in, below of=1]{$F(\G/\U)$};
\node(4)[node distance=.8in, below of=2]{$\displaystyle (\lim_{\mathclap{\substack{\to \\\V }}} F(\G/\V))^{\U}$};
\node(5)[node distance=.5in, right of=1]{$$};
\node(6)[node distance=.4in, below of=5]{$$};
\draw[->](1) to node [above]{$\eta$}(2);
\draw[->](1) to node [left]{$A$}(3);
\draw[->](2) to node [right]{$[-]\cdot A$}(4);
\draw[->](3) to node [below]{$\eta$}(4);
\end{tikzpicture}$$
commutes by the assumption that $\U\subset A^{-1}\mathscr WA$ holds.  The one compatibility condition for a 2-natural transformation also holds.  For this let $g\colon A\Rightarrow B$ denote a 2-cell of $\mathfrak S(\G)$ as in \ref{2-cell cond for 2-orbit cat} and \ref{well-definition 2-cells orbit cat}.  The equality of the composite 2-cells on the left and right sides
$$\begin{tikzpicture}
\node(1){$F(\G/\mathscr W)$};
\node(2)[node distance=1.5in, right of=1]{$\displaystyle (\lim_{\mathclap{\substack{\to \\\V }}} F(\G/\V))^{\mathscr W}$};
\node(3)[node distance=1in, below of=1]{$F(\G/\U)$};
\node(4)[node distance=1in, below of=2]{$\displaystyle (\lim_{\mathclap{\substack{\to \\\V }}} F(\G/\V))^{\U}$};
\node(5)[node distance=2in, right of=2]{$F(\G/\mathscr W)$};
\node(6)[node distance=1.5in, right of=5]{$\displaystyle (\lim_{\mathclap{\substack{\to \\\V }}} F(\G/\V))^{\mathscr W}$};
\node(7)[node distance=1in, below of=5]{$F(\G/\U)$};
\node(8)[node distance=1in, below of=6]{$\displaystyle (\lim_{\mathclap{\substack{\to \\\V }}} F(\G/\V))^{\U}$};
\node(9)[node distance=.42in, below of=1]{$g$};
\node(10)[node distance=.58in, below of=1]{$\Rightarrow$};
\node(11)[node distance=.42in, below of=6]{$-\cdot g$};
\node(12)[node distance=.58in, below of=6]{$\Rightarrow$};
\node(13)[node distance=2.5in, right of=9]{$=$};

\draw[->](1) to node [above]{$\eta$}(2);
\draw[->,bend right=40](1) to node [left]{$A$}(3);
\draw[->,bend left=40](1) to node [right]{$B$}(3);
\draw[->](2) to node [right]{$[-]\cdot B$}(4);
\draw[->](3) to node [above]{$\eta$}(4);

\draw[->](5) to node [above]{$\eta$}(6);
\draw[->](5) to node [left]{$A$}(7);
\draw[->,bend right=40](6) to node [left]{$[-]\cdot A$}(8);
\draw[->,bend left=40](6) to node [right]{$[-]\cdot B$}(8);
\draw[->](7) to node [above]{$\eta$}(8);
\end{tikzpicture}$$
follows by the constructions of Lemma \ref{fixed points is a functor on orbits}, the functoriality of $F$, and finally the definition of horizontal composition of 2-cells in $\mathfrak S(\G)$.  \end{proof}

The map $\eta$ of Lemma \ref{unit of 2-equivalence} is one-to-one on objects and faithful whenever $F$ is a 2-sheaf for the atomic topology.  This follows by the definition of the relations forming the colimit and by Lemma \ref{atomic top lemma} showing that each $FI$ is one-to-one on objects and faithful.  But in this case $\eta$ is a 2-natural isomorphism as well.  The proof of the following formal statement of the result shows surjectivity on objects and hom-wise on arrows.  This will suffice since a fully faithful and bijective-on-objects functor is an isomorphism of categories.  The argument in \S III.9 on p.154 of \cite{MM} can adapted for the proof.

\begin{lemma} \label{surjective when sheaf}  If $F\colon \mathfrak S(\G)^{op}\to\mathfrak{Cat}$ is a 2-sheaf for the atomic topology, then $\eta$ of Lemma \ref{unit of 2-equivalence} is surjective on objects.
\end{lemma}
\begin{proof}  To show surjectivity on objects, take $[\V,X]$ fixed by $\U$ under the action of $\G$.  Without loss of generality, assume that $\V\subset \U$ holds, so that the singleton
\[ \lbrace I\colon \G/\V\to \G/\U\rbrace
\]
generates a covering sieve $S$ on $\G/\U$ in the atomic topology as in Example \ref{atomic 2-topology}.  The claim is that $X\in F(\G/\V)_0$ defines compatible data $X\colon S\Rightarrow F$.  For $A\colon \G/\W\to\G/\U$ in $S$ factoring through $I\colon \G/\V\to \G/\U$ by say $B\colon \G/\W\to\G/\V$, take the component
\begin{equation} \label{obj to data} X_{\G/\W}(A):= X\cdot B
\end{equation}
On a morphism $g\colon A \Rightarrow A'$ factoring through $I\colon \G/\V\to \G/\U$ by some $h\colon B\Rightarrow B'$, take
\begin{equation}
\label{arrow to data} X_{\G/\mathscr W}(g):= h_X\colon X\cdot B\to X\cdot B'
\end{equation}
where $h_X$ denotes $F(h)_X$.  These assignments are functorial and will give a 2-natural transformation $X\colon S\Rightarrow F$ by the definition of horizontal composition of natural transformations \ref{horiz comp in cat}.  

However, \ref{obj to data} and \ref{arrow to data} need to be seen to be well-defined.  So, assume that $A$ factors through $I$ by another map say $C\colon \G/\W\to\G/\V$.  The question is whether $X\cdot B = X\cdot C$ holds.  For this, note that the assumption means that $C$ and $B$ define the same map $\G/\W\to \G/\U$, which implies that $C^{-1}\otimes B\in\U_0$.  Since $[\V,X]$ is $\U$-fixed, this means -- writing $D= C^{-1}\otimes B$ for readability --  that 
\[ [D^{-1}\V D,X\cdot D]=[\V,X]
\]
holds in the colimit.  But now from the definition of the relation and the fact that each $FI$ is one-to-one on objects by Lemma \ref{atomic top lemma}, it follows that $X\cdot B = X\cdot C$ does hold, proving well-definition.  The argument that \ref{arrow to data} is well-defined is similar.

Thus, by the 2-sheaf condition \ref{2-sheaf isomor}, there is an extension of the data $X\colon S\Rightarrow F$ along the inclusion $S\to \mathfrak S(\G)(-,\G/\U)$.  Its value at $I\colon \G/\U \to \G/\U$ is the required element of $F(\G/\U)$ mapping to $[\V,X]$ under $\eta$.  Consequently, $\eta$ is a bijection on objects.\end{proof}

\begin{remark}  Notice that there is no reason to believe that $I\colon \G/\V\to\G/\U$ is full.  This is why the atomic topology was defined the way it was in Example \ref{atomic 2-topology} and thus why the definition of a 2-sieve (i.e. Definition \ref{2-sieve defn} dropped Street's requirement of ``full."  See Remark \ref{Remark on Street}.
\end{remark}

\begin{lemma} \label{full when sheaf}  If $F\colon \mathfrak S(\G)^{op}\to\mathfrak{Cat}$ is a 2-sheaf for the atomic topology, then $\eta$ of Lemma \ref{unit of 2-equivalence} is full.
\end{lemma}
\begin{proof}  This argument follows a similar pattern.  That is, suppose that 
\[ [\V,m]\colon [\V,X] \to [\V,Y]
\]
is a morphism of the colimit, $\U$-fixed by the action of $\G$, and represented by $m\colon X\to Y$ in $F(\G/V)_1$.  The goal is to show that $m$ induces a morphism of compatible data $m\colon X\Rightarrow Y\colon S\Rightarrow F$ where $X$ and $Y$ are viewed as data $S\to F$ as in \ref{obj to data} and \ref{arrow to data} above.  The 2-dimensional aspect of the 2-sheaf condition \ref{2-sheaf isomor} will then guarantee an extension of $m$ between the extensions of $X$ and $Y$, yielding the element of $F(\G/\U)_1$ mapping to $[\V,m]$ under $\eta$.  So, the definition of the morphism $m\colon  X\Rrightarrow Y$ at say $A\colon \G/\W\to \G/\U$ factoring through $I\colon \G/\V\to \G/\U$ by say $B\colon \G/\W\to\G/\V$ is given on components by
\[ m_{\G/\W, B}:= m\cdot B\colon X\cdot B\to Y\cdot B
\]
in $F(\G\W)_1$.  Naturality in $B$ follows because $F(h)$ is natural for each 2-cell $h\colon B\Rightarrow B'$ of $\mathfrak S(\G)$.  Additionally, the collection of all $m_{\G/\mathscr W}$ will be a modification as in \ref{modification condition} by the axioms of the action of $\G$.  However, again the assignment needs to be shown to be well-defined, that is, independent of $B$ factoring $A$ through $I$.  But this part of the proof follows the pattern of the well-definition argument in the surjective-on-objects proof in Lemma \ref{surjective when sheaf} above using ``faithful" instead of ``injective on objects" from Lemma \ref{atomic top lemma}.  \end{proof}

\begin{theo} \label{MAIN THEOREM} With $\mathfrak S(\G)$ as in Construction \ref{underlying 2-cat construction}, there is a 2-equivalence
\[ \mathfrak{Sh}(\mathfrak S(\G),J_{at})\simeq \mathfrak B\G
\]
in the sense of Definition \ref{equiv of 2-cats} for any topological 2-group $\G$.
\end{theo}
\begin{proof}  All that remains is to check that the isomorphism $\Phi\Psi(F)\cong F$ of Propositions \ref{surjective when sheaf} and \ref{full when sheaf} is 2-natural in $F$ and $\X$ respectively.  But this follows from the definitions of the 2-functors.  \end{proof}

\section{Prospectus}

There are two ways in which Theorem \ref{MAIN THEOREM} is only a preliminary result.  The first is that it holds only for strict topological 2-groups.  Many interesting examples of 2-groups are the so-called ``coherent" 2-groups of \cite{Baez2Groups}.  The second is that the result depends upon very strict definitions coming from enriched category theory.  Insofar as this is the case, the result is not very ``bicategorical" and seems somewhat out of tune with the spirit of 2-dimensional category theory. 

Trying to address either one of these two deficiencies in Theorem \ref{MAIN THEOREM} seems to help with the other, however.  For in switching to a coherent topological 2-group $\G$, the weak monoidal structure on $\G$ will make the 2-orbit category $\mathfrak S(\G)$ into a bicategory since the operations are defined in terms of the monoidal structure.  This seems to necessitate a move to a more relaxed notion of a topology, namely, a ``bicovering" or a ``bitopology" rephrasing the notion of \S C2 in \cite{Elephant2} in an ``up-to-isomorphism" sense as detailed on the nLab \cite{nlab:2-sheaf}.  Of course this necessitates a move to so-called ``stacks" or ``bisheaves" instead of the 2-sheaves of \cite{StreetSheafThy} used in this paper.

On the other hand, certainly a 2-category is already a strict bicategory and there is nothing \emph{a priori} stopping a move to bicoverings and bisheaves on $\mathfrak S(\G)$ even in the case that $\G$ remains a strict 2-group.  The point of Lemma \ref{sheaf atomic top characterization} is that it is a well-definition criterion.  And this problem reappears in Propositions \ref{surjective when sheaf} and \ref{full when sheaf} whose proofs show that the preimage objects and arrows are constructed through defining coherence data whose well-definition is guaranteed by the facts that the $F\colon \mathfrak S(\G)^{op}\to\mathfrak{Cat}$ is a 2-sheaf and that $\G$ is a strict monoidal category.  Importantly, it seems that these well-definition problems are in fact posed by built-in strictness in the notions of sheaf and 2-sheaf.  It is not yet clear that stacks or bisheaves would pose the same, or at least same kind of, problem.  Insofar as this is the case, it might thus be possible to relax the monoidal structure on $\G$ as it might not be needed for this purpose.

As a final note, it appears that the real heart of the constructions in the classical result and the main result of the present paper are the colimit characterizations giving the object assignments of the functor $\mathbf BG\to [\mathbf S(G)^{op},\mathbf{Set}]$ and of the 2-functor $[\mathfrak S(\G)^{op},\mathfrak{Cat}] \to\mathfrak B\G$.  In particular these are so-called ``filtered" colimits which have a nice description in terms of germ-like equivalence classes under a tractable relation.  These characterizations are vital in showing that each colimit admits a continuous action of the group or 2-group, as the case may be.  Now, the colimit in the 2-case is actually just the 1-colimit of a category-valued functor, but it happens to have nice properties at the 2-cell level so that it gives a boosted up strict 2-colimit as well.  In switching to the bicategorical class of results for coherent 2-groups, it is expected that such boosted up 1-colimits will no longer be appropriate.  Rather it is expected that pseudo- or bi-colimits will have to be used.  And an explicit characterization of 2- or bi-filtered pseudo- and bi-colimits will be needed to canonically construct a well-defined continuous action.  For this purpose, however, there is the work of Descotte, Dubuc, and Szyld in \cite{DDS} and the work of the author's thesis \cite{LambertThesis}, which computes all weighted pseudo- and bi-colimits of categories and shows that those that are 2-filtered are formed by the process of the right calculus of fractions.

\bibliography{research}
\bibliographystyle{alpha}

\end{document}